\documentclass[a4paper,11pt]{amsart}

\usepackage{fancybox}
\usepackage{amscd}
\usepackage{amsmath}
\usepackage{amssymb}
\usepackage{amsthm}
\usepackage{float}
\usepackage[dvips, pdftex]{graphicx}
\usepackage{tikz}
\usepackage[all, ps, dvips, pdftex]{xy}
\usepackage{color}
\usepackage[hidelinks, hyperfootnotes=false]{hyperref}
\usepackage{array}
\usepackage{amscd}

\DeclareFontFamily{U}{rsfs}{%
\skewchar\font127}
\DeclareFontShape{U}{rsfs}{m}{n}{%
<-6>rsfs5<6-8.5>rsfs7<8.5->rsfs10}{}
\DeclareSymbolFont{rsfs}{U}{rsfs}{m}{n}
\DeclareSymbolFontAlphabet
{\mathrsfs}{rsfs}
\DeclareRobustCommand*\rsfs{%
\@fontswitch\relax\mathrsfs}

\theoremstyle{plain}
\newtheorem{thm}{Theorem}[section]
\newtheorem*{thm*}{Theorem}
\newtheorem{prop}[thm]{Proposition}
\newtheorem{constr}[thm]{Construction}
\newtheorem{lem}[thm]{Lemma}

\newtheorem{defi}[thm]{Definition}
\newtheorem{rmk}[thm]{Remark}

\newtheorem{prop-defi}[thm]{Proposition-Definition}
\newtheorem{thm-defi}[thm]{Theorem-Definition}
\newtheorem{lem-defi}[thm]{Lemma-Definition}

\newtheorem{question}[thm]{Question}
\newtheorem*{question*}{Question}

\newtheorem{setup-def}[thm]{Setup-Definition}

\newdimen\argwidth
\def\db[#1\db]{
 \setbox0=\hbox{$#1$}\argwidth=\wd0
 \setbox0=\hbox{$\left[\box0\right]$}
  \advance\argwidth by -\wd0
 \left[\kern.3\argwidth\box0 \kern.3\argwidth\right]}

\newcommand{\aA}{\mathcal{A}}
\newcommand{\bB}{\mathcal{B}}
\newcommand{\cC}{\mathcal{C}}
\newcommand{\dD}{\mathcal{D}}
\newcommand{\eE}{\mathcal{E}}
\newcommand{\fF}{\mathcal{F}}

\newcommand{\hH}{\mathcal{H}}

\newcommand{\kK}{\mathcal{K}}

\newcommand{\mM}{\mathcal{M}}
\newcommand{\nN}{\mathcal{N}}
\newcommand{\oO}{\mathcal{O}}

\newcommand{\wW}{\mathcal{W}}
\newcommand{\xX}{\mathcal{X}}

\newcommand{\zZ}{\mathcal{Z}}

\newcommand{\Ob}{\mathcal{O}b}
\newcommand{\OB}{\mathop{\rm Ob}\nolimits}
\newcommand{\Hone}{\mathop{H^1}\nolimits}
\newcommand{\bL}{\mathbb{L}}
\newcommand{\bQ}{\mathbb{Q}}

\newcommand{\fm}{\mathfrak{m}}

\newcommand{\fh}{\mathfrak{h}}

\newcommand{\fc}{\mathfrak{c}}
\newcommand{\fn}{\mathfrak{n}}

\renewcommand{\tilde}{\widetilde}

\newcommand{\lr}{\longrightarrow}

\newcommand{\Tor}{\mathop{\rm Tor}\nolimits}

\newcommand{\dR}{\mathbf{R}}

\newcommand{\id}{\textrm{id}}

\newcommand{\rk}{\mathop{\rm rk}\nolimits}

\newcommand{\Ext}{\mathop{\rm Ext}\nolimits}
\newcommand{\Spec}{\mathop{\rm Spec}\nolimits}

\newcommand{\Coh}{\mathop{\rm Coh}\nolimits}

\newcommand{\RHom}{\mathop{\dR\mathrm{Hom}}\nolimits}

\newcommand{\bA}{\mathbb{A}}
\newcommand{\bC}{\mathbb{C}}

\newcommand{\bP}{\mathbb{P}}

\def\lal{_\lambda}
\def\vir{\mathrm{\vir}}
\def\loc{\mathrm{\loc}}

\def\labc{_{\alpha \beta \gamma}}

\def\lra{\longrightarrow}
\def\ti{\tilde}

\def\lalp{_\alpha}
\def\sub{\subset}

\def\virt{^{\mathrm{vir}}}

\def\beq{\begin{equation}}
\def\eeq{\end{equation}}

\def\lalp{_\alpha }
\def\lab{_{\alpha\beta}}
\def\lbet{_\beta}

\def\rred{{\mathrm{red}}}
\def\loc{{\mathrm{loc}}}

\makeatletter
 
  \@addtoreset{equation}{section}
\makeatother

\makeatletter
\def\@tocline#1#2#3#4#5#6#7{\relax
  \ifnum #1>\c@tocdepth 
  \else
    \par \addpenalty\@secpenalty\addvspace{#2}%
    \begingroup \hyphenpenalty\@M
    \@ifempty{#4}{%
      \@tempdima\csname r@tocindent\number#1\endcsname\relax
    }{%
      \@tempdima#4\relax
    }%
    \parindent\z@ \leftskip#3\relax \advance\leftskip\@tempdima\relax
    \rightskip\@pnumwidth plus4em \parfillskip-\@pnumwidth
    #5\leavevmode\hskip-\@tempdima
      \ifcase #1
       \or\or \hskip 1em \or \hskip 2em \else \hskip 3em \fi%
      #6\nobreak\relax
    \hfill\hbox to\@pnumwidth{\@tocpagenum{#7}}\par
    \nobreak
    \endgroup
  \fi}
\makeatother

\title{K-theoretic Generalized Donaldson-Thomas Invariants}

\author{Young-Hoon Kiem}
\address{Department of Mathematical Sciences, Seoul National University, Seoul 08826, Korea}
\email{kiem@snu.ac.kr}

\author{Michail Savvas}
\address{Department of Mathematics, University of California, San Diego, La Jolla, CA 92093, USA}
\email{msavvas@ucsd.edu}

\thanks{YHK was partially supported by Samsung Science and Technology Foundation grant SSTF-BA1601-01. Part of this work was completed while MS was visiting the IHES, which he would like to thank for the excellent environment and working conditions.}

\begin{document}

\maketitle

\begin{abstract} 
We introduce the notion of almost perfect obstruction theory on a Deligne-Mumford stack and show that stacks with almost perfect obstruction theories have virtual structure sheaves which are deformation invariant.
The main components in the construction are an induced embedding of the coarse moduli sheaf of the intrinsic normal cone into the associated obstruction sheaf stack and the construction of a $K$-theoretic Gysin map for sheaf stacks.

We show that many stacks of interest admit almost perfect obstruction theories. As a result, we are able to define virtual structure sheaves and $K$-theoretic classical and generalized Donaldson-Thomas invariants of sheaves and complexes on Calabi-Yau threefolds.
\end{abstract}

\tableofcontents

\section{Introduction}

In enumerative geometry, 
one is concerned with finding the number of geometric objects satisfying a set of given conditions. Let $X$ be a scheme or more generally a Deligne-Mumford moduli stack which parametrizes the objects of interest. When $X$ is smooth and compact, it admits a fundamental cycle $[X] \in A_{\dim X}(X)$. One then obtains enumerative invariants by integrating appropriate cohomology classes against $[X]$, which are also invariant under suitable deformation of the counting problem.

However, in practice, $X$ is almost always not of the expected dimension, very singular and does not behave well under deformation. To deal with this problem, Li-Tian \cite{LiTian} and  Behrend-Fantechi \cite{BehFan} constructed the virtual fundamental cycle $[X]\virt \in A_{*}(X)$, which is of the expected dimension. This has been used to define many important enumerative invariants such as Gromov-Witten, Donaldson-Thomas \cite{Thomas} and Pandharipande-Thomas invariants \cite{PT1}.

Every Deligne-Mumford stack $X$ has an intrinsic normal cone $\cC_X$ which locally for an \'{e}tale morphism $U \to X$ and any embedding $U \to V$ into a smooth scheme $V$ is the quotient stack $[C_{U/V} / T_V|_U]$ of the normal cone of $U$ in $V$ by the tangent bundle $T_V$ of $V$ restricted to $U$ (cf. \cite{BehFan}). A perfect obstruction theory $\phi \colon E \to \bL_X^{\geq -1}$ (cf. Definition \ref{Perf obs th}) induces a closed embedding of $\cC_X$ into the vector bundle stack $\eE = h^1 / h^0 (E^\vee)$. One may then intersect $\cC_X$ with the zero section $0_{\eE}$ by using the Gysin map $0_{\eE}^!$ of a vector bundle stack (cf. \cite{Kresch}). The result is the virtual fundamental class  
\beq\label{y0}[X]\virt = 0_{\eE}^![\cC_X] \in A_* (X)\eeq
and integrating cohomology classes against $[X]\virt$ defines virtual invariants.

When $E$ admits a global presentation by vector bundles $[E^{-1} \to E^0]$, then we get a cone $C_1 = \cC_X \times_{\eE} E_1 \sub E_1$, where $E_1 = (E^{-1})^\vee$. Using the $K$-theoretic Gysin map $0_{E_1}^!$, one also obtains the virtual structure sheaf
\beq\label{y1} [\oO_X\virt] = 0_{E_1}^![\oO_{C_1}] \in K_0(X)\eeq
which can be viewed as a refinement of the virtual fundamental class $[X]\virt$. The $K$-theoretic virtual invariants are defined as the holomorphic Euler characteristic 
$\chi(X, [\oO_X\virt]\otimes \beta)$ for  $\beta\in K^0(X).$ 
\\

Recently there has been increased interest in moduli problems in which it is not clear how to construct a perfect obstruction theory, the most notable examples being moduli of perfect complexes \cite{Inaba, Lieblich} and generalized invariants where semistability and stability do not coincide \cite{KLS}. At the same time, motivated by applications to physics and geometric representation theory, it is desirable to refine the enumerative invariants beyond the level of intersection theory of cycles to $K$-theory. See \cite{Okou1, Okou2} for instance. So the following questions seem timely and interesting. 

\begin{question*}
Is there a virtual structure sheaf $[\oO_X\virt]\in K_0(X)$ when $X$ is a (not necessarily fine) component of the Inaba-Lieblich moduli space of simple universally gluable perfect complexes on a Calabi-Yau 3-fold or the partial desingularization of a moduli space of semistable objects? If yes, is the K-theoretic generalized Donaldson-Thomas invariant 
$$\chi(X,[\oO_X\virt]\otimes \beta), \quad \beta\in K^0(X)$$
defined by the virtual structure sheaf $[\oO_X\virt]$, invariant under deformation? 
\end{question*}

The purpose of this paper is to provide positive answers to the above questions 
and hence to generalize the K-theoretic Donaldson-Thomas invariants to derived category objects or partial desingularizations. 
The usual perfect obstruction theory in \cite{BehFan} is too strong a requirement since it doesn't seem to exist for many moduli spaces. On the other hand, the semi-perfect obstruction theory in \cite{LiTian, LiChang} (cf. Definition \ref{semi-perfect obs th}) seems too weak to 
guarantee a K-theory class for the obstruction cone. 
The novel technique we introduce here is the notion of  almost perfect obstruction theory (cf. Definition~\ref{almost perfect obs th}) which requires  less than the usual perfect obstruction theory but slightly more than the semi-perfect obstruction theory (cf. Proposition \ref{y4}, \eqref{y5}). 
We will see that all the moduli spaces for generalized Donaldson-Thomas invariants of derived category objects or their partial desingularizations admit almost perfect obstruction theories and that a Deligne-Mumford stack equipped with an almost perfect obstruction theory has a virtual structure sheaf by a recipe similar to \eqref{y1}.  \\ 

Roughly speaking, an almost perfect obstruction $\phi$ on a Deligne-Mumford stack $X$ consists of perfect obstruction theories $$\{\phi_\alpha:E_\alpha\lr  \bL_{U_\alpha}^{\geq -1}\}_{\alpha\in A}$$ for an \'etale cover $\{U_\alpha\to X\}_{\alpha\in A}$, 
whose obstruction sheaves $\{h^1(E_\alpha^\vee)\}$ glue to a coherent sheaf $Ob_\phi$ on $X$, such that for each pair $\alpha,\beta\in A$, we have isomorphisms between $E_\alpha|_{U_{\alpha\beta}}$ and $E_\beta|_{U_{\alpha\beta}}$ over $ \bL_{U_{\alpha\beta}}^{\geq -1}$  \'etale locally where $U_{\alpha\beta}=U_\alpha\times_XU_\beta$. 
We will see that an almost perfect obstruction theory $\phi$ on $X$ induces a closed embedding of the coarse moduli sheaf $\fc_X$ of the intrinsic normal cone into the obstruction sheaf $\Ob_\phi$. We also construct a $K$-theoretic Gysin map $0_{\Ob_\phi}^!$ for the sheaf stack $\Ob_\phi$. 
The virtual structure sheaf $[\oO_X\virt]$ is now defined as the result of applying $0_{\Ob_\phi}^!$ to 
the structure sheaf $[\oO_{\fc_X}]$ of $\fc_X$. 

Even in the case where $X$ admits a perfect obstruction theory, our method improves the construction of the virtual structure sheaf $[\oO_X\virt]$ in \cite{BehFan, yplee} in that we no longer need to assume that the perfect obstruction theory $E$ admits a resolution by a globally defined two-term complex of locally free sheaves.

We summarize the main results of this paper as follows.

\begin{thm*}
Let $X \to S$ be a morphism with an almost perfect obstruction theory (cf. Definition~\ref{almost perfect obs th}), where $X$ is a Deligne-Mumford stack of finite presentation and $S$ a smooth quasi-projective scheme. Then the coarse intrinsic normal cone stack $\fc_X$ embeds into the obstruction sheaf $\Ob_\phi$ (cf. Theorem~\ref{thm 3.4}) and using the $K$-theoretic Gysin map $0_{\Ob_\phi}^!$ (cf. Definition~\ref{Gysin morphism}) we may define the virtual structure sheaf of $X$ as
$$[\oO_X\virt] = 0_{\Ob_\phi}^! [\oO_{\fc_X}] \in K_0(X) $$
which is deformation invariant (cf. Theorem~\ref{deformation invariance}). 
\end{thm*}

Almost perfect obstruction theories are flexible enough to appear in many moduli problems of interest and in particular moduli of sheaves and complexes on Calabi-Yau threefolds. They can thus be used to construct virtual structure sheaves and associated $K$-theoretic generalized Donaldson-Thomas invariants. 
In this paper, we discuss the following applications to Donaldson-Thomas theory:
\begin{enumerate}
    \item Gieseker or slope semistable sheaves.
    \item Simple perfect complexes. 
    \item PT-semistable complexes.
    \item Bridgeland semistable complexes.
\end{enumerate}

\subsection*{Layout of the paper} In \S 2, we develop a formalism for $K$-theory of coherent sheaves for sheaf stacks over Deligne-Mumford stacks and define a $K$-theoretic Gysin map. In \S 3, we introduce almost perfect obstruction theories and show that they induce an embedding of the coarse intrinsic normal cone into the obstruction sheaf stack. In \S 4, we combine the results of \S 2 and \S 3 to construct virtual structure sheaves for stacks with almost perfect obstruction theories and prove their deformation invariance. \S 5 focuses on applications, including virtual structure sheaves for derived and d-critical stacks and $K$-theoretic classical and generalized Donaldson-Thomas invariants of sheaves and complexes. Finally, in the Appendix, we prove that an almost perfect obstruction theory is a semi-perfect obstruction theory while a perfect obstruction theory is an almost perfect obstruction theory. 

\subsection*{Notation and conventions} Everything in this paper is over the field $\bC$ of complex numbers. All stacks are of finite type. Deligne-Mumford stacks are separated. $X \to S$ will typically denote a morphism of finite type, where $X$ is a Deligne-Mumford stack and $S$ a smooth curve or more generally a smooth Artin stack, locally of finite type and of pure dimension. When we write $U \to S$ we will typically refer to an \'{e}tale cover of $X \to S$ equipped with a perfect obstruction theory, which will often be part of the data of a semi-perfect or almost perfect obstruction theory on $X \to S$.

If $E$ is a locally free sheaf on a Deligne-Mumford stack $X$, we will use the term ``vector bundle" to refer to its total space. If $\fF$ is a coherent sheaf on a Deligne-Mumford stack $X$, we will use the same letter to refer to the associated sheaf stack.

DM stands for Deligne-Mumford and DT for Donaldson-Thomas.

\section{$K$-Theoretic Gysin Maps on Sheaf Stacks} \label{sec2}

Let $X$ be a Deligne-Mumford stack and $\fF$ a coherent sheaf on $X$, with $0_\fF$ being its zero section. 

\begin{defi} \emph{(Sheaf stack)} The \emph{sheaf stack} associated to $\fF$ is the stack that to every morphism $\rho \colon W \to X$ from a scheme $W$ associates the set $\Gamma(W, \rho^\ast \fF)$. 
\end{defi}
 
By abuse of notation we denote by $\fF$ the sheaf stack associated to a coherent sheaf $\fF$ on $X$.  

In this section, we develop a formalism for $K$-theory of coherent sheaves on $\fF$ and define an associated Gysin map 
$$0_\fF^! \colon K_0(\fF) \lr K_0(X).$$

These generalize the usual definitions when $\fF$ is a vector bundle. The main point is that $\fF$ is in general not algebraic, so we may not work with smooth covers by schemes. However, this role will be played by morphisms of the form $r_E \colon E \to \fF|_U$, where $U$ is a scheme with an \'{e}tale map $U \to X$, $E$ is a locally free sheaf on $U$ and $r_E$ is a surjection. 

\subsection{Local charts for $\fF$} The preceding discussion motivates the following definition.

\begin{defi} \emph{(Local chart)} \label{local chart}
Let $U$ be a scheme, $\rho \colon U \to X$ an \'{e}tale morphism, $E$ be a locally free sheaf on $U$ and $r_E \colon E \to \rho^* \fF=\fF|_U$ be a surjection. We say that the data $(U, \rho, E, r_E)$ give a \emph{local chart} on $\fF$ with base $\rho \colon U \to X$ and vector bundle $E$ and denote the composition $E \to \fF|_U \to \fF$ by $\bar{r}_E$. If $U$ is affine and $E$ is free, then we refer to the data $Q = (U, \rho, E, r_E)$ as an affine local chart.
\end{defi}

We also define morphisms between local charts.

\begin{defi} \emph{(Morphism between local charts)}
Let $Q=(U, \rho, E, r_E)$ and $Q'=(U', \rho', E', r_{E'})$ be two local charts for $\fF$. A \emph{morphism} $\gamma \colon Q \to Q'$ is the data $(\rho_\gamma, r_\gamma)$ of an \'{e}tale morphism $\rho_\gamma \colon U \to U'$ and a surjection $r_\gamma \colon E \to \rho_\gamma^* E'$ such that the triangles
\begin{align*}
\xymatrix{
U \ar[r]^-{\rho_\gamma} \ar[dr]_-{\rho} & U' \ar[d]^-{\rho'} \\
& X
} \ \textrm{} \ \xymatrix{
E \ar[r]^-{r_\gamma} \ar[dr]_-{r_E} & \rho_\gamma^* E' \ar[d]^-{\rho_\gamma^* r_{E'}} \\
& \fF|_U
}
\end{align*}
are commutative.

We say that $Q$ is a restriction of $Q'$ and write $Q = Q'|_U$ if $E = \rho_\gamma^* E'$ and $r_\gamma$ is the identity morphism.
\end{defi}

\subsection{$K_0(\fF)$ for a sheaf stack $\fF$} In what follows, $X$ is a Deligne-Mumford stack and $\fF$ is a sheaf stack over $X$ associated to a coherent sheaf $\fF$. 

By a coherent sheaf $\aA$ on $\fF$ we mean an assignment to every local chart $Q = (U, \rho, E, r_E)$ of a coherent sheaf $\aA_Q$ on the scheme $E$ (in the \'{e}tale topology) such that for every morphism $\gamma \colon Q \to Q'$ between local charts we have an isomorphism 
\beq\label{y2} r_\gamma^* \left( \rho_\gamma^*\aA_{Q'} \right) \lr \aA_Q\eeq
that satisfies the usual compatibilities for composition of morphisms. Note that we abusively write $\rho_\gamma^* \aA_{Q'}$ for the pullback of $\aA_{Q'}$ to $\rho_\gamma^* E'$ via the morphism of bundles $\rho_\gamma^* E' \to E'$ induced by $\rho_\gamma$.

\begin{rmk} \label{structure sheaf of closed substack}
Any closed substack $Z \sub \fF$ has a structure sheaf $\oO_Z$ which assigns to every local chart $Q = (U, \rho, E, r_E)$ the sheaf $\oO_{Z \times_\fF E}$, where $E$ maps to $\fF$ via $\bar{r}_E$. It is easy to see that $\oO_Z$ is a coherent sheaf on $\fF$.
\end{rmk}

A homomorphism $f:\aA\to \bB$ of coherent sheaves on $\fF$ refers to a homomorphism
$f_Q:\aA_Q\to \bB_Q$ of coherent sheaves on $E$ for each local chart $Q=(U,\rho,E,r_E)$ such that for every morphism $\gamma:Q\to Q'$ of local charts,
the diagram 
$$\xymatrix{
r_\gamma^* \left( \rho_\gamma^*\aA_{Q'} \right) \ar[r]\ar[d]_{f_{Q'}}& \aA_Q\ar[d]^{f_Q}\\
r_\gamma^* \left( \rho_\gamma^*\bB_{Q'} \right) \ar[r] & \bB_Q
}$$
is commutative where the horizontal arrows are \eqref{y2}. We say that a homomorphism $f:\aA\to \bB$ is an isomorphism if $f_Q$ is an isomorphism for each local chart $Q$.

We can define the notion of an exact sequence.
\begin{defi} \emph{(Short exact sequence)} \label{ses}
Let $\aA, \bB, \cC$ be coherent sheaves on $\fF$. A sequence 
\begin{align*}
    0 \lr \aA \lr \bB \lr \cC \lr 0
\end{align*}
of homomorphisms of coherent sheaves on $\fF$  
is \emph{exact} if for every local chart $Q=(U, \rho, E, r_E)$ on $\fF$ the sequence
\begin{align*}
    0 \lr \aA_Q \lr \bB_Q \lr \cC_Q \lr 0
\end{align*}
is an exact sequence of coherent sheaves on the scheme $E$.
\end{defi}

We may now define the $K$-group of coherent sheaves on $\fF$ in the usual way.

\begin{defi} 
The \emph{$K$-group} of coherent sheaves on $\fF$ is the group generated by the isomorphism classes $[\aA]$ of coherent sheaves $\aA$ on $\fF$, with relations generated by $[\bB] = [\aA] + [\cC]$ for every short exact sequence $$0 \lr \aA \lr \bB \lr \cC \lr 0.$$
\end{defi}

\begin{rmk}
If $\fF$ is locally free, then $\fF$ is an algebraic stack and the above definitions recover the usual notions of short exact sequences and $K_0(\fF)$. This is because the morphism $\bar{r}$ associated to a local chart is smooth, as $r$ is a surjection of locally free sheaves and hence a smooth morphism on their total spaces, and therefore a collection of local charts that covers $\fF$ will give a smooth atlas for $\fF$.
\end{rmk}

\subsection{Koszul homology} Let $(U, \rho, E, r_E)$ be a local chart for $\fF$ and denote the vector bundle projection map $E \to U$ by $\pi_E$. Then the tautological section of the pullback $\pi_E^* E$ induces an associated Koszul complex $\wedge^\bullet \pi_E^* E^\vee$ that resolves the structure sheaf $\oO_U$ of the zero section of $\pi_E$.

\begin{defi}
We define $K(E)$ to be the above Koszul complex $\wedge^\bullet \pi_E^* E^\vee$.
\end{defi}

\begin{defi}
For any $Q = (U, \rho, E, r_E)$ and coherent sheaf $\aA$ on $\fF$, the \emph{$i$-th Koszul homology sheaf} $\hH_Q^i(\aA)$ of $\aA$ with respect to $Q$ is defined as the homology of the complex $\wedge^\bullet \pi_E^* E^\vee \otimes_{\oO_E} \aA_Q$ in degree $-i$.
\end{defi}

\begin{lem}
$\hH_Q^i(\aA)$ is an $\oO_U$-module isomorphic to $\Tor^{\oO_E}_i(\oO_U, \aA_Q)$.
\end{lem}

\begin{proof} 
It is a standard property of Koszul homology that $\hH_Q^i(\aA)$ is an $\oO_U$-module. Since $\wedge^\bullet \pi_E^* E^\vee$ is a resolution of $\oO_U$, $\hH_Q^i(\aA)$ is isomorphic to the homology of $\oO_U \otimes_{\oO_E}^L \aA_Q$ in degree $-i$ which is $\Tor_{\oO_E}^i(\oO_U, \aA_Q)$ by definition. 
\end{proof} 

The next proposition shows that a morphism between local charts induces a morphism on the associated Koszul homology sheaves.

\begin{prop} \label{morphism of local charts}
Let $\gamma \colon Q \to Q'$ be a morphism between two local charts. Then for any coherent sheaf $\aA$ on $\fF$ there exists an induced isomorphism $h^i(\gamma) \colon \rho_\gamma^* \hH_{Q'}^i(\aA) \to \hH_Q^i(\aA) $ of $\oO_U$-modules. If $\gamma':Q'\to Q''$ is another morphism of local charts, we have
\beq\label{y3} h^i(\gamma'\circ \gamma)=h^i(\gamma)\circ h^i(\gamma').\eeq
\end{prop}

\begin{proof} Let us assume first that $Q$ and $Q'$ are affine local charts, so that $U = \Spec B,\ U' = \Spec A$ are affine and $E,\ E'$ are free modules.

By the definition, using the injection of rings $\oO_{\rho_\gamma^* E'} \to \oO_E$, we get a morphism of complexes of quasicoherent $\oO_{\rho_\gamma^* E'}$-modules
\begin{align*}
    K(\rho_\gamma^* E') \lr K(E).
\end{align*}
Thus we obtain a morphism
\begin{align} \label{morloc}
    K(\rho_\gamma^* E') \otimes_{\oO_{\rho_\gamma^* E'}} \rho_\gamma^* \aA_{Q'} \lr K(E) \otimes_{\oO_{\rho_\gamma^* E'}} \rho_\gamma^* \aA_{Q'}.
\end{align}

Now $r_\gamma^*( \rho_\gamma^* \aA_{Q'} ) $ is naturally isomorphic to $\aA_Q$ by the commutativity requirements in the definition of $\gamma$ and so we have the isomorphism
\begin{align*}
  \oO_E \otimes_{\oO_{\rho_\gamma^* E'}} \rho_\gamma^* \aA_{Q'}  \cong r_\gamma^*( \rho_\gamma^* \aA_{Q'} ) \cong \aA_Q.
\end{align*}
This allows us to express the right hand side of \eqref{morloc} as
\begin{align*}
 K(E) \otimes_{\oO_{\rho_\gamma^* E'}} \rho_\gamma^* \aA_{Q'} \cong K(E) \otimes_{\oO_E} \oO_E \otimes_{\oO_{\rho_\gamma^* E'}} \rho_\gamma^* \aA_{Q'} \cong K(E) \otimes_{\oO_E} \aA_Q
\end{align*}
giving an $\oO_{\rho_\gamma^* E'}$-linear morphism
\begin{align} \label{morloc2}
    K(\rho_\gamma^* E') \otimes_{\oO_{\rho_\gamma^* E'}} \rho_\gamma^* \aA_{Q'} \lr K(E) \otimes_{\oO_E} \aA_Q
\end{align}

One may check that this is a quasi-isomorphism by splitting the exact sequence
\begin{align*} 
    0 \lr R \lr E \xrightarrow{r_\gamma} \rho_\gamma^* E' \lr 0
\end{align*}
Since the morphism $\rho_\gamma$ is flat, the homology of the left hand side of \eqref{morloc2} computes $\rho_\gamma^* \hH_{Q'}^i(\aA)$ while the right hand side gives $\hH_Q^i(\aA)$. Thus we obtain an isomorphism 
\begin{align*}
    \rho_\gamma^* \hH_{Q'}^i(\aA) \to \hH_Q^i(\aA) 
\end{align*}

For the general case, we may cover $U$ and $U'$ by open affine subschemes to obtain affine local charts. It is routine to check that the maps obtained agree on overlaps giving the desired isomorphism.

The equality \eqref{y3} is straightforward to check and we omit it. 
\end{proof}

We now introduce a way to compare two local charts on $\fF$ with the same base $\rho \colon U \to X$.

\begin{defi} \emph{(Common roof)}
Let $Q=(U, \rho, E, r_E)$ and $Q'=(U, \rho, E', r_{E'})$ be two local charts with the same base $\rho \colon U \to X$. Let $W \to E \times_{\fF|_U} E'$ be a surjection, where $W$ is a locally free sheaf on $U$ and $E \times_{\fF|_U} E'$ denotes the kernel of the morphism 
\begin{align*} 
   E \oplus E' \xrightarrow{(r_E,-r_{E'})} \fF|_U \oplus \fF|_U \xrightarrow{+} \fF|_U 
\end{align*}
so that we have a commutative diagram
\begin{align*}
    \xymatrix{
     & W \ar[d] & \\
     & E \times_{\fF|_U} E' \ar[dl] \ar[dr] & \\
     E \ar[dr]_-{r_E} & & E' \ar[dl]^-{r_{E'}} \\
     & \fF|_U &
    }
\end{align*}
where all the arrows are surjective. 

Denote the induced surjection $W \to \fF|_U$ by $r_W$. The local chart $R=(U, \rho, W, r_W)$ is a \emph{common roof} for the charts $Q=(U, \rho, E, r_E)$ and $Q'=(U, \rho, E', r_{E'})$. There are natural morphisms of local charts $\gamma \colon R \to Q$ and $\gamma' \colon R \to Q'$.
\end{defi}

\begin{rmk} \label{restriction of roofs}
If $R$ is a roof for two charts $Q, Q'$ as above and $V \to U$ is \'{e}tale, then we may restrict (pullback) the roof $R$ to $V$ to obtain a common roof between the charts $Q|_V$ and $Q' |_V$. 

By definition, $E \times_{\fF|_U} E'$ fits in an exact sequence
\begin{align*}
    0 \lr E \times_{\fF|_U} E' \lr E \oplus E' \xrightarrow{r_E - r_{E'}} \fF|_U \lr 0.
\end{align*}
Since $V \to U$ is flat, pulling back to $V$ gives an exact sequence
\begin{align*}
    0 \lr \left( E \times_{\fF|_U} E' \right)|_V \lr E|_V \oplus E'|_V \xrightarrow{r_E|_V - r_{E'}|_V} \fF|_V \lr 0
\end{align*}
implying that 
$$\left( E \times_{\fF|_U} E' \right) |_V \cong E|_V \times_{\fF|_V} E'|_V.$$

Hence we can pull back the surjection $W \to E \times_{\fF|_U} E'$ to obtain a surjection $W|_V \to E|_V \times_{\fF|_V} E'|_V$ inducing a roof, which is the restriction of $R$ to $V$ and denoted by $R|_V$.
\end{rmk}

Given a coherent sheaf $\aA$ on $\fF$, local charts $Q, Q'$ as above and a roof $R$, we can define the comparison isomorphism
\begin{align*}
    h_R^i := h^i(\gamma') \circ h^i(\gamma)^{-1} \colon \hH_Q^i(\aA) \lr \hH_{Q'}^i(\aA)
\end{align*}

\begin{lem} \label{independence from roof}
The comparison isomorphism $h_R^i$ does not depend on the choice of roof $R$.
\end{lem}

\begin{proof}
Suppose that we have two roofs $R_1, R_2$ induced by two surjections $W_1 \to E \times_{\fF|_U} E'$ and $W_2  \to E \times_{\fF|_U} E'$. Since this is a local statement, we may assume that all the local charts are affine.

Let $W_3$ be a locally free sheaf fitting in a commutative diagram
\begin{align*}
    \xymatrix{
    & W_3 \ar[dl] \ar[dr] & \\
     W_1 \ar[dr] & & W_2 \ar[dl] \\
     & E \times_{\fF|_U} E' &
    }
\end{align*}
where all the arrows are surjective. We obtain an induced local chart $R_3$ with morphisms to $R_1$ and $R_2$.

Using \eqref{y3}, it is now a simple diagram chase with roofs to verify that 
$$ h_{R_1}^i = h_{R_3}^i = h_{R_2}^i $$
which is what we want.
\end{proof}

Suppose now that $\aA$ is a coherent sheaf on the sheaf stack $\fF$. Using the above, we may globalize the Koszul homology sheaves as follows.

\begin{constr} \label{construction} Let $\lbrace Q\lalp = (U\lalp, \rho\lalp, E\lalp, r_{E\lalp}) \rbrace_{\alpha \in A}$ be a collection of affine local charts so that the morphisms $\rho\lalp \colon U\lalp \to X$ give an \'{e}tale cover of $X$. We write $U\lab = U\lalp \times_X U\lbet$ and $U\labc = U\lalp \times_X U\lbet \times_X U_\gamma$. \\

(a) For each $\alpha$, we obtain the $i$-th Koszul homology sheaf $\bB\lalp := \hH_{Q\lalp}^i(\aA)$, which is a coherent $\oO_{U\lalp}$-module. \\
    
(b) For any two indices $\alpha, \beta$, the restrictions $\bB\lalp|_{U\lab} = \hH_{Q\lalp}^i(\aA)|_{U\lab}$ and $\bB\lbet|_{U\lab}=\hH_{Q\lbet}^i(\aA)|_{U\lab}$ are naturally isomorphic to $\hH_{Q\lalp|_{U\lab}}^i(\aA)$ and $\hH_{Q\lbet|_{U\lab}}^i(\aA)$ respectively. The same is true for any further restriction to an \'{e}tale open $V\lal \to U\lab$. \\
    
(c) We may now construct a canonical comparison isomorphism $$g\lab \colon \bB\lalp|_{U\lab} \to \bB\lbet|_{U\lab}$$
using roofs and Lemma \ref{independence from roof}. 

Let $\lbrace V\lal \rbrace_{\lambda \in \Lambda}$ be a cover of $U\lab$ by Zariski open affine subschemes. For any $\lambda$, since $V\lal$ is affine, there exists a roof $R\lal$ for the restrictions $Q\lalp|_{V\lal}$ and $Q\lbet|_{V\lal}$, induced by a surjection 
$$W\lal \lr E\lalp|_{V\lal} \times_{\fF|_{V\lal}} E\lbet|_{V\lal}$$
where $W\lal$ is a free sheaf on $V\lal$. Using this roof,  
$$ g\lal := h_{R\lal}^i \colon \bB\lalp |_{V\lal} \lr \bB\lbet |_{V\lal}.$$
For any two indices $\lambda, \mu \in \Lambda$, we write $V_{\lambda \mu}:= V_\lambda \times_{U\lab} V_\mu$. The restrictions $g\lal|_{V_{\lambda \mu}}$ and $g_\mu |_{V_{\lambda \mu}}$ are the comparison isomorphisms induced by the roofs $R\lal|_{V_{\lambda \mu}}$ and $R_\mu |_{V_{\lambda \mu}}$ (see Remark~\ref{restriction of roofs}) and hence by Lemma~\ref{independence from roof} we must have $g\lal|_{V_{\lambda \mu}} = g_\mu|_{V_{\lambda \mu}}$. By \'{e}tale descent the collection of morphisms $\lbrace g\lal \rbrace$ glues to give a comparison isomorphism $g\lab$ as desired.\\
    
(d) The comparison isomorphisms $g\lab$ of part (c) satisfy the cocycle condition. Up to \'{e}tale shrinking, the composition $g_{\beta \gamma}|_{U\labc} \circ g\lab|_{U\labc}$ is induced by two roofs $R\lab$ over $Q\lalp|_{U\labc}$ and $Q\lbet|_{U\labc}$ and $R_{\beta \gamma}$ over $Q\lbet|_{U\labc}$ and $Q_\gamma|_{U\labc}$. Similarly, the isomorphism $g_{\alpha \gamma}|_{U\labc}$ is induced by a roof $R_{\alpha \gamma}$ over $Q\lalp|_{U\labc}$ and $Q_\gamma|_{U\labc}$. As in the proof of Lemma~\ref{independence from roof}, we may find (up to further shrinking) a common roof $R_{\alpha \gamma}'$ over $R\lab$ and $R_{\beta \gamma}$ so that $g_{\beta \gamma}|_{U\labc} \circ g\lab|_{U\labc}$ is the isomorphism induced by $R_{\alpha \gamma}'$. But then Lemma~\ref{independence from roof} again implies that $$g_{\beta \gamma}|_{U\labc} \circ g\lab|_{U\labc} = g_{\alpha \gamma}|_{U\labc}.$$ \\
    
(e) It follows that the sheaves $\bB\lalp$ descend to a sheaf $\hH^i(\aA)$. It is a standard check to verify that this sheaf is independent of the particular choice of collection of affine local charts in part (a) that cover $X$ by taking the union of any two such collections and showing that the sheaves are canonically isomorphic. We leave the details to the reader.
\end{constr}

\begin{defi} \label{Koszul homology sheaves}
Let $\aA$ be a coherent sheaf on a sheaf stack $\fF$ on a Deligne-Mumford stack $X$. The sheaf $\hH_K^i(\aA) \in \Coh(X)$ is defined to be the \emph{$i$-th Koszul homology sheaf} of $\aA$.
\end{defi}

\subsection{$K$-theoretic Gysin map} We are now ready to define a $K$-theoretic operation of intersecting with the zero section $0_\fF$ of a sheaf stack $\fF$.

\begin{defi} \emph{($K$-theoretic Gysin map)}
The \emph{$K$-theoretic Gysin map} of a sheaf stack $\fF$ is a homomorphism 
$$0_\fF^! \colon K_0(\fF) \lr K_0(X)$$ 
such that for any coherent sheaf $\aA$ on $\fF$,  we have
\begin{align} \label{Gysin morphism}
    0_\fF^! [\aA] = \sum_{i \geq 0} (-1)^i [\hH_K^i (\aA)] \in K_0(X).
\end{align}
\end{defi}

\begin{prop}
The Gysin map $0_\fF^! \colon K_0(\fF) \to K_0(X)$ is well-defined.
\end{prop}

\begin{proof}
Firstly, the sum in \eqref{Gysin morphism} is finite, since for any cover of $X$ by finitely many affine local charts $\lbrace Q\lalp = (U\lalp, \rho\lalp, E\lalp, r_{E\lalp}) \rbrace_{\alpha \in A}$ it is clear by the definitions that $\hH_K^i (\aA) = 0$ for any $i > \max_{\alpha \in A} \rk E\lalp$.

Now we need to check that for any short exact sequence (cf. Definition~\ref{ses})
\begin{align*}
    0 \lr \aA \lr \bB \lr \cC \lr 0
\end{align*}
we have 
\begin{align} \label{2.4}
    0_\fF^![\bB] = 0_\fF^![\aA] + 0_\fF^![\cC].
\end{align}
Let $Q = (U, \rho, E, r_E)$ be a local chart for $\fF$. Then we obtain an exact sequence of coherent sheaves on $E$
\begin{align*}
    0 \lr \aA_Q \lr \bB_Q \lr \cC_Q \lr 0
\end{align*}
which induces an exact triangle of complexes
\begin{align*}
    K(E) \otimes_{\oO_E} \aA_Q \lr K(E) \otimes_{\oO_E} \bB_Q \lr K(E) \otimes_{\oO_E} \cC_Q \lr K(E) \otimes_{\oO_E} \aA_Q [1]
\end{align*}
and thus a long exact sequence in cohomology
\begin{align*}
    ... \lr \hH^i_Q (\aA) \lr \hH^i_Q (\bB) \lr \hH^i_Q (\cC) \lr \hH^{i+1}_Q (\aA) 
    \lr ...
\end{align*}
These long exact sequences are functorial with respect to morphisms $\gamma \colon Q \to Q'$ between local charts and hence, as in Construction~\ref{construction}, we get a long exact sequence of Koszul homology sheaves
\begin{align*}
    ... \lr \hH^i_K (\aA) \lr \hH^i_K (\bB) \lr \hH^i_K (\cC) \lr \hH^{i+1}_K (\aA) 
    \lr ...
\end{align*}
which immediately implies \eqref{2.4} by the definition of Gysin morphism \eqref{Gysin morphism}. \end{proof}

\begin{rmk}
If there exists a surjection $r \colon E \to \fF$ where $E$ is a locally free sheaf on $X$, for example when $X$ has the resolution property, then it follows immediately by the definition that for any coherent sheaf $\aA$ on $\fF$ we have the equality
\begin{align*}
    0_\fF^![\aA] = 0_E^! [\aA_Q] = [\aA_Q \otimes_{\oO_E}^L \oO_X] \in K_0(X)
\end{align*}
where $Q = (X, \id, E, r)$.

Our definition is obviously consistent with the usual Gysin morphism for vector bundles when $\fF$ is locally free by taking $E=\fF$ and $r=\id$.
\end{rmk}

\begin{rmk}
The Gysin map constructed here should be viewed as the $K$-theoretic analogue of the Chow theoretic Gysin map $0_\fF^!$ constructed in \cite{LiChang}.
\end{rmk}

\section{Almost Perfect Obstruction Theory}

As usual, let $X \to S$ be a Deligne-Mumford stack of finite type over a smooth Artin stack of pure dimension. If $X$ admits a semi-perfect obstruction theory $\phi$ (cf. Appendix~\ref{appendix}), then by the results of \cite{LiChang} there exists an intrinsic normal cone cycle $[\fc_\phi] \in Z_* \Ob_\phi$ in the associated sheaf stack $\Ob_\phi$ whose intersection with the zero section of $\Ob_\phi$ defines the virtual fundamental class of $X$. 

In this section, we define the notion of an almost perfect obstruction theory, which is stronger than that of a semi-perfect obstruction theory but still weaker than that of a perfect obstruction theory. This turns out to be an appropriate intermediary notion which allows us to define an intrinsic normal cone stack $\fc_\phi \sub \Ob_\phi$ in the sheaf stack $\Ob_\phi$. This arises in most known natural examples (cf. \S 5) and will be used in \S 4 to define a virtual structure sheaf.

\subsection{Almost perfect obstruction theory} We begin by giving the definition. See Definition~\ref{Perf obs th} for the definition of a perfect obstruction theory. 

\begin{defi} \emph{(Almost perfect obstruction theory)} \label{almost perfect obs th}
Let $X \to S$ be a morphism, where $X$ is a DM stack of finite presentation and $S$ is a smooth Artin stack of pure dimension. An \emph{almost perfect obstruction theory} $\phi$ consists of an \'{e}tale covering $\lbrace U_\alpha \to X \rbrace_{\alpha \in A}$ 
of $X$ and perfect obstruction theories $\phi_\alpha \colon E_\alpha \to \bL_{U_\alpha / S}^{\geq -1}$ of $U_\alpha$ such that the following hold. 
\begin{enumerate}
\item For each pair of indices $\alpha, \beta$, there exists an isomorphism \begin{align*}
\psi_{\alpha \beta} \colon \Ob_{\phi_\alpha} \vert_{U_{\alpha\beta}} \lra \Ob_{\phi_\beta} \vert_{U_{\alpha\beta}}
\end{align*}
so that the collection $\lbrace \Ob_{\phi\lalp}=h^1(E_\alpha^\vee), \psi\lab \rbrace$ gives descent data of a sheaf $\Ob_\phi$, called the obstruction sheaf, on $X$.
\item For each pair of indices $\alpha, \beta$, there exists an \'{e}tale covering $\lbrace V_\lambda \to U\lab \rbrace_{\lambda \in \Gamma}$ of $U\lab=U_\alpha\times_XU_\beta$ such that for any $\lambda$, the perfect obstruction theories $E_\alpha \vert_{V_{\lambda}}$ and $E_\beta \vert_{V_{\lambda}}$ are isomorphic and compatible with $\psi\lab$. This means that there exists an isomorphism
\begin{align*}
    \eta_{\alpha\beta\lambda} \colon E_\alpha \vert_{V_{\lambda}} \lr E_\beta \vert_{V_{\lambda}}
\end{align*} 
in $D^b(\Coh V_\lambda)$ fitting in a commutative diagram
\begin{align} \label{loc 5.1}
    \xymatrix{
    E_\alpha \vert_{V_{\lambda}} \ar[d]_-{\phi\lalp|_{V_\lambda}} \ar[r]^-{ \eta_{\alpha\beta\lambda}} & E_\beta \vert_{V_{\lambda}} \ar[d]^-{\phi\lbet|_{V_\lambda}} \\
    \bL_{U_\alpha / S}^{\geq -1}|_{V_\lambda} \ar[r] \ar[dr] & \bL_{U_\beta / S}^{\geq -1}|_{V_\lambda} \ar[d] \\
    & \bL_{V_\lambda / S}^{\geq -1}
    }
\end{align}
which moreover satisfies $h^1(\eta_{\alpha\beta\lambda}^\vee) = \psi\lab^{-1}|_{V_\lambda}$.
\end{enumerate}
\end{defi}

\begin{rmk}
We note that the isomorphisms $\eta_{\alpha\beta\lambda}$ of the definition are not required to satify any compatibility relations, in contrast to the isomorphisms $\psi\lab{}$.
\end{rmk}

An almost perfect obstruction theory is in particular a semi-perfect obstruction theory by the following proposition, which shows that part (2) of Definition \ref{almost perfect obs th} is a strengthening of the second condition in the definition of a semi-perfect obstruction theory (cf. Definition~\ref{semi-perfect obs th}).

\begin{prop}\label{y4}
An almost perfect obstruction theory on $X \to S$ naturally induces a semi-perfect obstruction theory.
\end{prop}
See the appendix for a proof.  Certainly the usual perfect obstruction theory of \cite{BehFan} is 
an almost perfect obstruction theory with the \'etale cover $\id:X\to X$. We thus have 
\beq\label{y5} \text{POT} \Rightarrow \text{almost POT} \Rightarrow \text{semi-POT}\eeq
where $\text{POT}$ stands for perfect obstruction theory.

\subsection{The coarse intrinsic normal cone stack} Suppose that $X \to S$ admits an almost perfect obstruction theory.  Let $\fn_{U\lalp/S}=h^1((\bL_{U\lalp/S}^{\geq -1})^\vee)$ denote the coarse intrinsic normal sheaf, which we think of as a sheaf stack. 
For each pair of indices $\alpha, \beta$, consider the diagram
\begin{align} \label{loc 5.2}
\xymatrix{
    \fn_{U\lbet/S}|_{U\lab} \ar[r] \ar[d]_-{h^1(\phi\lbet^\vee)} &  \fn_{U\lalp/S}|_{U\lab} \ar[d]^-{h^1(\phi\lalp^\vee)}\\
    \Ob_{\phi\lbet}|_{U\lab} \ar[r]_-{\psi\lab^{-1}} &  \Ob_{\phi\lalp}|_{U\lab}
}
\end{align}
where the top horizontal arrow is the natural isomorphism. The restriction of \eqref{loc 5.2} to $V_\lambda \to U\lab$ is commutative, since $\psi\lab^{-1}|_{V_\lambda} = h^1(\eta_{\alpha\beta\lambda}^\vee)$ and the diagram \eqref{loc 5.1} commutes. But $\{V_\lambda \to U\lab\}_\lambda$ give an \'{e}tale cover of $U\lab$ and therefore \eqref{loc 5.2} commutes.

We deduce that the closed embeddings 
$$h^1(\phi\lalp^\vee) \colon \fn_{U\lalp/S} \lr \Ob_{\phi\lalp}$$
glue to a global closed embedding
$$j_\phi \colon \fn_{X/S} \lr \Ob_{\phi}$$
of sheaf stacks over $X$ where $\fn_{X/S}=h^1((\bL_{X/S}^{\geq -1})^\vee)$. 
By \cite{BehFan}, the coarse intrinsic normal cone stack $\fc_{X/S}$ (resp. $\fc_{U_\alpha/S}$) is a closed substack of the intrinsic normal sheaf stack $\fn_{X/S}$ (resp. $\fn_{U_\alpha/S}$) and hence $\fc_\phi = j_\phi(\fc_{X/S})$ (resp. $\fc_{\phi_\alpha} = h^1(\phi\lalp^\vee)(\fc_{U_\alpha/S})$) is a closed substack of $\Ob_\phi$ (resp. $\Ob_{\phi_\alpha}$). We have thus established the following theorem.

\begin{thm} \label{thm 3.4}
Let $X \to S$ be a morphism, where $X$ is a Deligne-Mumford  stack of finite presentation and $S$ a smooth Artin stack of pure dimension. Let $\phi$ be an almost perfect obstruction theory on $X \to S$. Then there exists a closed cone substack $\fc_\phi \sub \Ob_\phi$ such that for any \'{e}tale $U\lalp \to X$ we have $\fc_\phi |_{U\lalp} = \fc_{\phi\lalp}$ via the obvious natural identifications.
\end{thm}

\section{Virtual Structure Sheaves for Almost Perfect Obstruction Theories} \label{sec 5}

Let $X \to S$ be a morphism, where $X$ is a Deligne-Mumford  stack of finite type and $S$ a smooth Artin stack of pure dimension, together with an almost perfect obstruction theory $\phi$. In this section, we combine the results of the two preceding sections to construct a virtual structure sheaf $[\oO_X\virt] \in K_0(X)$ and show that it is deformation invariant.

\subsection{Virtual structure sheaves} We first recall the definition of a virtual structure sheaf when $U \to S$ has a perfect obstruction theory $\phi \colon E \to \bL_{U/S}^{\geq -1}$ (cf. Definition~\ref{Perf obs th}).

Let us assume that $E$ has a global resolution 
\begin{align*}
E = [ E^{-1} \lr E^0 ],
\end{align*}
where $E^{-1}, E^0$ are locally free sheaves on $U$. We denote $E_i = \left( E^{-i} \right)^\vee$ for $i=0,1$.

Then $\eE = h^1/h^0(E^\vee) = [ E_1 / E_0 ]$ so that we have the quotient morphism $E_1 \to \eE$. By \cite[Proposition 2.2]{BehFun}, the diagram
\begin{align} \label{curvilinear}
\xymatrix{
C_1 \ar[r] \ar[d] & E_1 \ar[d] \\
\cC_{U/S} \ar[r] \ar[d] & \eE \ar[d] \\
\fc_{U/S} \ar[r] & \Ob_\phi,
}
\end{align}
is Cartesian and gives rise to the obstruction cone $C_1$. The virtual structure sheaf associated to the perfect obstruction theory $\phi$ is defined as
\begin{align*} 
[\oO_U\virt, \phi] & = [\oO_U \otimes^L_{\oO_{E_1}} \oO_{C_1}] = [K(E_1) \otimes_{\oO_{E_1}} \oO_{C_1} ]  \\
& =\sum_i (-1)^i [ \Tor^{\oO_{E_1}}_{i}(\oO_U, \oO_{C_1}) ] \in K_0(U)
\end{align*}
where $K(E_1)$ denotes the Koszul resolution of $\oO_U$. 

Using our definition of Gysin map for the sheaf stack $\Ob_\phi$, we see that the virtual structure sheaf is equal to
\begin{align*} 
[\oO_U\virt, \phi] = 0_{E_1}^! [\oO_{C_1}] = 0_{\Ob_\phi}^! [\oO_{\fc_{U/S}}].
\end{align*}
So the $K$-theoretic Gysin maps for sheaf stacks constructed in \S 2 enable us to drop the requirement in \cite{BehFan} and \cite{yplee} that the perfect obstruction theory $E$ should admit a global resolution by locally free sheaves. \\

In the case of $X \to S$ with an \emph{almost perfect obstruction theory} $\phi$, the closed substack $\fc_\phi \sub \Ob_\phi$ gives rise to a coherent sheaf $\oO_{\fc_\phi} \in \Coh(\Ob_\phi)$ (cf. Remark~\ref{structure sheaf of closed substack}). We may thus give the following definition using the Gysin map $0_{\Ob_\phi}^!$, in analogy with the above.

\begin{defi}
The \emph{virtual structure sheaf} of $X \to S$ associated to the almost perfect obstruction theory $\phi$ is defined as
\begin{align*}
   [\oO_X\virt, \phi] = 0_{\Ob_\phi}^!  [\oO_{\fc_\phi}]  \in K_0(X).
\end{align*}
\end{defi}
We will often denote the virtual structure sheaf $ [\oO_X\virt, \phi]$ by $[\oO_X\virt]$ for simplicity.

\subsection{Deformation invariance} \label{deformation invariance} Suppose that we have a Cartesian diagram
\begin{align*}
\xymatrix{
    Y \ar[r]^-{u} \ar[d] & X \ar[d] \\
    Z \ar[r]_-{v} & W
    }
\end{align*}
where $Z,W$ are smooth varieties and $v$ is a regular embedding. Let $\phi$ be an almost perfect obstruction theory on $X \to S$, given by perfect obstruction theories $\phi\lalp \colon E\lalp \to \bL_{U\lalp / S}^{\geq -1}$ on an \'{e}tale cover $\lbrace U\lalp \to X \rbrace_{\alpha \in A}$ of $X$. Let
\begin{align*}
    \xymatrix{
    V\lalp \ar[r]^-{u\lalp} \ar[d] & U\lalp \ar[d] \\
    Y \ar[r]_-{u} & X
    }
\end{align*}
be Cartesian. Suppose now that we have an almost perfect obstruction theory on $Y \to S$ given by perfect obstruction theories $\phi'\lalp \colon E\lalp' \to \bL_{V\lalp/S}^{\geq -1}$ together with commutative diagrams
\begin{align} \label{compat obs th}
\xymatrix{
    E\lalp|_{V\lalp} \ar[r]^-{g\lalp} \ar[d]_-{\phi\lalp|_{V\lalp}} & E\lalp' \ar[d]^-{\phi\lalp'} \ar[r] & N_{Z/W}^\vee |_{V\lalp}[1] \ar@{=}[d] \ar[r] &\\
    \bL_{U\lalp/S}^{\geq -1}|_{V\lalp} \ar[r] & \bL_{V\lalp / S}^{\geq -1} \ar[r] & \bL_{V\lalp/U\lalp}^{\geq -1} \ar[r] &
    }
\end{align}
of distinguished triangles which are compatible with the diagrams \eqref{loc 5.1} for $\phi$ and $\phi'$ such that we have exact sequences 
\begin{align}
    N_{Z/W}|_{V\lalp} \lr \Ob_{\phi'\lalp} \xrightarrow{h^1(g\lalp^\vee)} \Ob_{\phi\lalp}|_{V\lalp} \lr 0.
\end{align}
that glue to a sequence 
\begin{align}
    N_{Z/W}|_{Y} \lr \Ob_{\phi'} \lr \Ob_{\phi}|_{Y} \lr 0.
\end{align}

\begin{thm}
$[\oO_Y\virt, \phi'] = v^! [\oO_X\virt, \phi] \in K_0(Y)$.
\end{thm}

Here the Gysin map $v^! \colon K_0(X) \to K_0(Y)$ is defined by the formula
\begin{align}
    v^! [\aA] = [\oO_Z^W |_X \otimes_{\oO_X} \aA] \in K_0(Y)
\end{align}
where we fix $\oO_Z^W$ to be a finite locally free resolution of $v_* \oO_Z$. By \cite{yplee}, $v^!$ also equals the composition
\begin{align}
   K_0(X) \xrightarrow{\sigma_u} K_0(C_{Y/X}) \xrightarrow{0_{N_{Z/W}}^!} K_0(Y)
\end{align}
where $\sigma_u$ is specialization to the normal cone and $0_{N_{Z/W}}^!$ is the Gysin map induced from the Cartesian diagram
\begin{align*}
\xymatrix{
    Y \ar[d] \ar[r] & C_{Y/X} \ar[d] \\
    Z \ar[r] & N_{Z/W}
    }
\end{align*}

\begin{proof}[Proof of Theorem 4.3]
This is a standard argument in the context of functoriality of virtual cycles in intersection theory, modified appropriately in the $K$-theoretic setting. We give an outline, leaving the details to the reader.

Let $\mM_X^\circ \to \bP^1$ be the deformation of $X$ to its intrinsic normal cone stack $\cC_X$. Then we define $\wW = \mM_{Y \times \bP^1 / \mM_X^\circ}^\circ$ to be the double deformation space given by the deformation of $Y \times \bP^1$ inside $\mM_X^\circ$ to its normal cone $\cC_{Y \times \bP^1 / \mM_X^\circ}$. We have a morphism $\wW \to \bP^1 \times \bP^1$. We denote the two projections $\wW \to \bP^1$ by $\pi_1$ and $\pi_2$ respectively. 

The fiber over $(1,0)$ is $\cC_Y$ while the flat specialization at the point $(0,0)$ along $\lbrace 0 \rbrace \times \bP^1$ is $\cC_{Y / \cC_X}$. We have $[\bC_0] = [\bC_1] \in K_0(\bP^1)$ and thus using the projection $\pi_1$
\begin{align} \label{loc 5.3}
    [\oO_{\wW} \otimes_{\oO_{\bP^1}}^L \bC_0] = [\oO_{\wW} \otimes_{\oO_{\bP^1}}^L \bC_1].
\end{align}

$\pi_1$ is flat over $\bP^1 - \lbrace 0 \rbrace$ and hence the right hand side is equal to
\begin{align} \label{loc 5.4}
    [\oO_{\wW} \otimes_{\oO_{\bP^1}}^L \bC_1] =  [\oO_{\pi_1^{-1}(1)}] = [\oO_{\mM_Y^\circ}].
\end{align}

Let us denote by $D$ the Cartier divisor $\cC_{Y\times \bP^1 / \mM_X^\circ}$ inside $\wW$. We then have an operation of ``intersecting with $D$" 
\begin{align*}
    D \cdot (\bullet) \colon K_0(\wW) \lr K_0(D)
\end{align*}
defined by the formula
\begin{align*}
    D \cdot [\aA] = [\oO_D \otimes^L_{\oO_\wW} \aA] = [D \otimes_{\oO_\wW} \aA]
\end{align*}
where by abuse of notation we denote $D = [ \oO_{\wW}(-D) \to \oO_{\wW}]$. 

Since $\pi_2$ is flat and $D = \pi_2^{-1}(0)$, this has the property that for any closed substack
$$\zZ^\circ \sub \wW^\circ := \pi_2^{-1}\left( \bP^1 - \lbrace 0 \rbrace \right)$$
flat over $\bP^1 - \lbrace 0 \rbrace$ and any class $\aA \in K_0(\wW)$ such that $\aA|_{\wW^\circ} = [\oO_\zZ^\circ] \in K_0(\wW^\circ)$ we have
$$D \cdot \aA = [\oO_{\zZ_0^{\mathrm{fl}}}]$$
where $\zZ_0^{\mathrm{fl}}$ is the flat specialization of $\zZ^\circ$.

Since $\mM_Y^\circ$ is flat over $\bP^1$ via the projection $\pi_2$ with fiber $\cC_Y$ over $0$ we have 
\begin{align} \label{loc 5.5}
   D \cdot [\oO_{\mM_Y^\circ}] = [\oO_{\cC_Y}].
\end{align}

Moreover
\begin{align} \label{loc 5.6}
    D \cdot [\oO_{\wW} \otimes_{\oO_{\bP^1}}^L \bC_0] = [\oO_{\cC_{Y/\cC_X}}]
\end{align} 
since $[\oO_{\wW} \otimes_{\oO_{\bP^1}}^L \bC_0] \in K_0(\wW)$ restricts to $[\oO_{\wW^\circ \cap \pi_1^{-1}(0)}] \in K_0(\wW^\circ)$ and $\wW^\circ\cap \pi_1^{-1}(0)$ specializes to $\cC_{Y / \cC_X}$.

Combining \eqref{loc 5.3}, \eqref{loc 5.4}, \eqref{loc 5.5} and \eqref{loc 5.6}, we obtain
\begin{align}
   [\oO_{\cC_Y}] =  [\oO_{\cC_{Y/\cC_X}}] \in K_0(\cC_{Y\times \bP^1 / \mM_X^\circ})
\end{align}

Since $\cC_{Y \times \bP^1/\mM_X^\circ}$ is a closed substack of $\nN_{Y \times \bP^1 / \mM_X^\circ}$ the equality holds in $K_0(\nN_{Y \times \bP^1 / \mM_X^\circ})$ as well.

Following \cite{KimKreschPant}, for each index $\alpha$ we consider the commutative diagram of distinguished triangles on $V\lalp \times \bP^1$
\begin{align}
\xymatrix{
E\lalp|_{V\lalp}(-1) \ar[r]^-{\kappa\lalp} \ar[d] & E\lalp|_{V\lalp} \oplus E\lalp' \ar[r] \ar[d] & c(\kappa\lalp) \ar[r] \ar[d] &\\
\bL_{U\lalp/S}^{\geq -1} |_{V\lalp}(-1) \ar[r]_-{\lambda\lalp} & \bL_{U\lalp/S}^{\geq -1} |_{V\lalp} \oplus \bL_{V\lalp/S}^{\geq -1} \ar[r] & c(\lambda\lalp) \ar[r] &
}
\end{align}
where $\kappa\lalp = (T \cdot \id, U \cdot g\lalp)$ with $T,U$ coordinates on $\bP^1$.

Clearly $\lambda\lalp$ is the restriction to $V\lalp$ of a global morphism $\lambda$. By \cite{KimKreschPant}, we have that $h^1/h^0(c(\lambda)^\vee) = \nN_{Y \times \bP^1 / \mM_X^\circ}$.

By the properties of almost perfect obstruction theories and the compatibility diagrams \eqref{compat obs th}, the closed embeddings
$$h^1(c(\lambda\lalp)^\vee) \lr h^1(c(\kappa\lalp)^\vee)$$
glue to a closed embedding of sheaf stacks on $Y \times \bP^1$
$$\fn_{Y\times \bP^1 / \mM_X^\circ} \lr \kK$$

The same argument as above works at the level of coarse moduli sheaves, where flatness stands for exactness of the pullback functor. Thus we deduce the equality
\begin{align} \label{loc 6.11}
    [\oO_{\fc_Y}] =  [\oO_{\fc_{Y/\cC_X}}] \in K_0(\kK)
\end{align}

The fiber of $\kK$ over $\lbrace 0 \rbrace \in \bP^1$ is $\Ob_{\phi}|_Y \oplus N_{Z/W}$ while the fiber over $\lbrace 1 \rbrace \in \bP^1$ is $\Ob_{\phi'}$. Therefore, we obtain by \eqref{loc 6.11}
\begin{align*}
    [\oO_Y\virt, \phi'] = 0_{\Ob_{\phi'}}^! [\oO_{\fc_Y}] = 0_{\Ob_{\phi}|_Y \oplus N_{Z/W}|_Y}^! [\oO_{\fc_{Y/\cC_X}}]
\end{align*}
Now, since the usual properties of Gysin maps hold by working on local charts of the corresponding sheaf stacks, we have
\begin{align*}
    0_{\Ob_{\phi}|_Y \oplus N_{Z/W}}^! [\oO_{\fc_{Y/\cC_X}}] = 0_{\Ob_\phi |_Y}^! 0_{N_{Z/W}|_Y}^! [\oO_{\fc_{Y/\cC_X}}] = 0_{\Ob_\phi |_Y}^! v^! [\oO_{\fc_X}]
\end{align*}
By the next proposition, we have $0_{\Ob_\phi|_Y}^! v^! = v^! \Ob_\phi^!$, which implies the desired equality.
\end{proof}

\begin{prop}
$0_{\Ob_\phi|_Y}^! v^! = v^! \Ob_\phi^! \colon K_0(\Ob_\phi) \to K_0(Y)$.
\end{prop}

\begin{proof}
For any coherent sheaf $\aA$ on $\Ob_\phi$, we have
\begin{align} \label{loc 6.14}
    0_{\Ob_\phi|_Y}^! v^! [\aA] = \sum (-1)^{i+j} \left[ \hH_K^i \left( \hH^j \left( \aA \otimes \oO_Z^W \right) \right) \right]
\end{align}
and
\begin{align} \label{loc 6.15}
    v^! \Ob_\phi^! [\aA] = \sum (-1)^{i+j} \left[ \hH^i \left( \hH_K^j ( \aA ) \otimes \oO_Z^W \right) \right]
\end{align}

For any local chart $Q = (U, \rho, E, r)$ we have a homology sheaf $$\hH^\ell ( K(E) \otimes \aA_E \otimes \oO_Z^W|_E  )$$ on $U \times_X Y$. By an identical argument as in Construction~\ref{construction}, these glue to define sheaves $\bB^\ell$ on $Y$. 

Considering the spectral sequence for the double complex $K(E) \otimes \left( \aA_E \otimes \oO_Z^W|_E  \right)$ on each local chart $Q = (U, \rho, E, r)$ with second page given by $\hH_Q^i \left( \hH^j \left( \aA \otimes \oO_Z^W \right) \right)$, it is easy to see that for each $\ell$
\begin{align} \label{loc 6.16}
    \sum_\ell (-1)^\ell [ \bB^\ell ] = \sum_{i+j = \ell} (-1)^{i+j} \left[ \hH_K^i \left( \hH^j \left( \aA \otimes \oO_Z^W \right) \right) \right]
\end{align}
since the spectral sequences are functorial with respect to morphisms of local charts $\gamma \colon Q \to Q'$.

Similarly for the double complex $(K(E) \otimes \aA_E) \otimes \oO_Z^W|_E $ we get
\begin{align} \label{loc 6.17}
    \sum_\ell (-1)^\ell [ \bB^\ell ] = \sum_{i+j = \ell} (-1)^{i+j} \left[ \hH^i \left( \hH_K^j ( \aA ) \otimes \oO_Z^W \right) \right]
\end{align}
Combining \eqref{loc 6.14}, \eqref{loc 6.15}, \eqref{loc 6.16} and \eqref{loc 6.17} yields the desired equality.
\end{proof}

Having defined the virtual structure sheaf, we may raise the following natural question.
\begin{question}
Do torus localization \cite{GrabPand} and cosection localization \cite{Cosection} hold for the virtual structure sheaf under an almost perfect obstruction theory?
\end{question}
We will get back to this question in a subsequent paper. 

\section{Donaldson-Thomas Invariants and other Applications}

In this section, we discuss several applications of the theory developed in this paper. Before we do so, we introduce some terminology for convenience and state and prove a lemma, that will be used in multiple occasions.

\begin{defi}[Kuranishi model] A \emph{Kuranishi model} for a scheme $U \to S$ is the data of a triple $\Lambda = (V, F_V, \omega_V)$ where $V \to S$ is a smooth morphism, $F_V$ is a locally free sheaf on $V$ and $\omega_V \in H^0(V, F_V)$ such that the vanishing locus of $\omega_V$ is precisely $U$.

A Kuranishi model induces a perfect obstruction theory on $U \to S$
\begin{align*}
    \xymatrix{
    E_\Lambda \ar@{=}[r] \ar[d] & [F_V^\vee |_U \ar[d]_-{\omega_V^\vee} \ar[r]^-{d\omega_V^\vee} & \Omega_{V/S}|_U] \ar@{=}[d] \\
    \bL_{U/S}^{\geq -1} \ar@{=}[r] & [I/I^2 \ar[r]_-{d} & \Omega_{V/S}|_U]
    }
\end{align*}
where $I$ is the ideal sheaf of $U$ in $V$.
\end{defi}

\begin{defi}
Let $K = (V,F_V,\omega_V)$ and $\Lambda = (W, F_W, \omega_W)$ be two Kuranishi models on schemes $T \to S$ and $U \to S$ respectively. We say that $K, \Lambda$ are \emph{$\Omega$-compatible} if they satisfy:
\begin{enumerate}
    \item There exist an \'{e}tale morphism $T \to U$ and an unramified morphism $\Phi \colon V \to W$ such that the diagram
    \begin{align*}
    \xymatrix{
        T \ar[d] \ar[r] & U \ar[d] \\
        V \ar[dr] \ar[r]^-{\Phi} & W \ar[d] \\
         & S
        }
    \end{align*}
    commutes.
    \item There exists a surjective morphism $\eta_\Phi \colon F_W |_V \to F_V$ such that $$\eta_\Phi(\omega_W |_V)=\omega_V.$$
    \item $\eta_\Phi$ induces an isomorphism of obstruction sheaves 
    $$\eta_\Phi \colon h^1(E_K^\vee) \lr h^1(E_\Lambda^\vee|_T).$$
\end{enumerate}
\end{defi}

\begin{lem} \label{lemma 7.2}
Let $K = (V,F_V,\omega_V)$ and $\Lambda = (W, F_W, \omega_W)$ be two $\Omega$-compatible Kuranishi models on $S$-schemes $T$ and $U$ respectively. Then, up to shrinking $T$, there exists a quasi-isomorphism $\psi \colon E_\Lambda|_T \to E_K$ making the triangle
\begin{align*}
    \xymatrix{
    E_\Lambda|_T \ar[d] \ar[r]^-{\psi} & E_K \ar[d] \\
    \bL_{T/S}^{\geq -1}|_U \ar[r] & \bL_{U/S}^{\geq -1}
    }
\end{align*}
commutative and satisfying $h^1(\psi^\vee) = \eta_\Phi \colon h^1(E_K^\vee) \lr h^1(E_\Lambda^\vee|_T)$.
\end{lem}

\begin{proof}
Let $I$ be the ideal of $U$ in $V$ and $J$ be the ideal of $U$ in $W$. Up to shrinking $V$, we may split the surjection $\eta_\Phi$ and assume that $F_W|_V = F_V \oplus R $ such that $\eta_\Phi$ is given by projection to the first factor. We then have a commutative diagram
\begin{align*}
    \xymatrix{
    F_V^\vee |_T \ar[d] \ar[r] & F_V^\vee|_T \oplus R^\vee |_T \ar[d] \\
    I/I^2 \ar[d] & \ar[l] J/J^2|_T \ar[d] \\
    \Omega_{V/S}|_T & \ar[l] \Omega_{W/S}|_T
    }
\end{align*}
Up to shrinking $T$, since $F_V^\vee |_T \to I/I^2$ is surjective, we may lift the arrow $R^\vee|_T \to I/I^2$ to a morphism $\alpha \colon R^\vee |_T \to F_V^\vee |_T$. It is then simple to check that the commutative diagram
\begin{align*}
    \xymatrix{
    F_V^\vee |_T \ar[d] & \ar[l]^-{\id + \alpha} F_V^\vee|_T \oplus R^\vee |_T \ar[d] \\
    I/I^2 \ar[d] & \ar[l] J/J^2|_T \ar[d] \\
    \Omega_{V/S}|_T & \ar[l] \Omega_{W/S}|_T
    }
\end{align*}
gives the desired quasi-isomorphism $\psi$.
\end{proof}

\subsection{Derived Deligne-Mumford stacks} \label{derived application} Let $\xX$ be a quasi-smooth derived Deligne-Mumford stack with classical truncation the Deligne-Mumford stack $X = t_0(\xX)$. The restriction $\bL_\xX|_X$  of its derived cotangent complex to $X$ is a perfect complex with Tor-amplitude $[-1,0]$ and the morphism
$$\bL_\xX|_X \lr \bL_X^{\geq -1}$$
gives a perfect obstruction theory $\phi$ on $X$. Thus the coarse intrinsic normal cone stack $\fc_X$ embeds into the sheaf stack $\Ob_\phi$ and we may define a virtual structure sheaf
$$[\oO_X\virt,\phi] = 0^!_{\Ob_\phi}[\oO_{\fc_X}] \in K_0(X).$$
This should coincide with the usual structure sheaf 
$$[\oO_X\virt] = \sum_{i \leq 0} (-1)^i [\pi_i(\oO_\xX)] \in K_0(X).$$
This follows from [CFK] in the particular case when $\xX$ is a quasi-smooth dg-scheme.

\subsection{d-critical Deligne-Mumford stacks} \label{d-crit application} Let $X$ be a d-critical Deligne-Mumford stack (cf. \cite{JoyceDCrit}) or a critical virtual manifold (cf. \cite{KiemLiCat}). 
By \cite{JoyceDCrit}, we have an \'{e}tale cover $\lbrace U\lalp \to X \rbrace_{\alpha \in A}$ with the following properties:
\begin{enumerate}
    \item For each $\alpha$, there exists a smooth scheme $V\lalp$ and a function $f\lalp \colon V\lalp \to \bA^1$ such that $\Lambda\lalp = (V\lalp, \Omega_{V\lalp}, df\lalp)$ is a Kuranishi model for $U\lalp$, called a d-critical chart, inducing a perfect obstruction theory $\phi\lalp \colon E\lalp \to \bL_{U\lalp}^{\geq -1}$.
    \item For every pair of indices $\alpha, \beta$, there exists an \'{e}tale cover $\lbrace T_\gamma \to U\lab \rbrace_{\gamma \in \Gamma}$ such that for $\lambda = \alpha, \beta$ there exist unramified morphisms $\Phi_\lambda' \colon V\lal' \to V\lal$ making the diagrams
    \begin{align*}
    \xymatrix{
        T_\gamma \ar[d] \ar[r] & U\lal \ar[d] \\
        V\lal' \ar[r] & V\lal
        }
    \end{align*}
    commute and $K\lal = (V\lal', \Omega_{V\lal'}, f\lal|_{V\lal'})$ is a d-critical chart on $T_\gamma$.
    \item There exists a d-critical chart $M_\gamma = (W_\gamma, \Omega_{W_\gamma}, df_\gamma)$ for $T_\gamma$ and unramified morphisms $\Phi\lal \colon V\lal' \to W_\gamma$ such that $f_\gamma \circ \Phi\lal = f\lal|_{V\lal'}$.
\end{enumerate}

By (3), we see that $K\lal$ and $M_\gamma$ are $\Omega$-compatible and thus by Lemma~\ref{lemma 7.2} $E_{K\lal}$ and $E_{M_\gamma}$ are isomorphic obstruction theories on $T_\gamma$. By (2), each $E_{K\lal}$ is isomorphic as a perfect obstruction theory with $E\lalp|_{T_\gamma}$. Combining these two, we see that $E\lalp|_{T_\gamma}$ and $E\lbet|_{T_\gamma}$ are isomorphic obstruction theories. Moreover, by the results of \cite{JoyceDCrit} it follows that the induced isomorphisms at the level of obstruction sheaves satisfy the cocycle condition and thus glue to define a global obstruction sheaf on $X$. We therefore obtain an almost perfect obstruction theory $\phi$ for $X$ on the cover $\lbrace U\lalp \to X \rbrace$. Note that each obstruction theory $E\lalp$ is symmetric and $\Ob_\phi = \Omega_X$. We deduce the following theorem.

\begin{thm}
Let $X$ be a d-critical Deligne-Mumford stack. Then $X$ admits an almost perfect obstruction theory $\phi$ and thus has a virtual structure sheaf $[\oO_X\virt] = 0^!_{\Ob_\phi}[\oO_{\fc_X}] = 0^!_{\Omega_X} [\oO_{\fc_X}] \in K_0(X)$.
\end{thm}

\begin{rmk}
If $X$ is the truncation of a $(-1)$-shifted symplectic derived Deligne-Mumford stack $\xX$, then $E\lalp \simeq \bL_{\xX}|_{U\lalp}$ and the virtual structure sheaf agrees with the one constructed in Subsection~\ref{derived application}.
\end{rmk}

\subsection{$K$-theoretic Donaldson-Thomas invariants of simple perfect complexes} \label{Classical DT application} The results of this paper apply to the main application of \cite{LiChang} as well, yielding $K$-theoretic Donaldson-Thomas invariants of simple bounded complexes. We recall the setup.

Let $\pi \colon X \to S$ be a smooth, proper family of Calabi-Yau threefolds over a smooth base $S$. Fix a line bundle $L$ on $X$ and let $\dD^L_{X/S}$ be the moduli space of simple universally gluable  perfect complexes $E \in D^b(\Coh X)$ whose determinant is isomorphic to $L$. Here a perfect complex  $E$ is called universally gluable if $\Ext^{<0}(E,E)=0$. 
More precisely, $\dD^L_{X/S}$ sends each $S$-scheme $T$ to the set of simple universally gluable perfect complexes $E \in D^b(\Coh (T \times_S X))$ such that $\det E \cong \pi_X^* L \otimes \pi_T^* J$, for some line bundle $J$ on $T$ with $\pi_S, \pi_T$ the projections of $T \times_S X$ onto its two factors.
The existence of such a moduli space of an algebraic space, locally of finite type, follows from \cite{Inaba, Lieblich}. 

Let $\mM \sub \dD^L_{X/S}$ be a proper, open and closed subspace. Then, using the existence of a universal semi-family, \cite{LiChang} produce the following data:
\begin{enumerate}
    \item An \'{e}tale cover $\lbrace U\lalp \to \mM \rbrace_{\alpha \in A}$.
    \item For each index $\alpha$, a perfect complex $E\lalp \in D^b(\Coh (X \times_S U\lalp))$.
    \item For each pair of indices $\alpha$ and $\beta$, quasi-isomorphisms
        \begin{align} \label{6.1}
            f\lab \colon E\lalp |_{X \times_S U\lab} \lr E\lbet  |_{X \times_S U\lab}
        \end{align}
        which satisfy the cocycle condition: for any triple of indices $\alpha, \beta, \gamma$ there exists a $c\labc \in \Gamma(\oO_{U\labc}^*)$ such that
        \begin{align} \label{cocycle 1}
            \pi_{\gamma \alpha}^* f_{\gamma \alpha} \circ  \pi_{\beta \gamma}^* f_{\beta \gamma} \circ \pi\lab^* f\lab = c\labc \cdot \id \colon E\lalp |_{X \times_S U\labc} \lr E\lalp  |_{X \times_S U\labc}
        \end{align}
        where $\pi\lab \colon X \times_S U\labc \to X \times_S U\lab$ denotes the projection and similarly for $\pi_{\beta \gamma}, \pi_{\gamma \alpha}$.
\end{enumerate}

Using the Atiyah class as in \cite{HuyThomas} and the properties of universal semi-families, \cite{LiChang} proceed to construct for each index $\alpha$, a perfect obstruction theory 
\begin{align*}
    \phi\lalp \colon \RHom(E\lalp, E\lalp)_0[1]^\vee \lr \bL_{U\lalp/S}^{\geq -1}
\end{align*}
where the subscript $0$ denotes the traceless part. We denote $$F\lalp = \RHom(E\lalp, E\lalp)_0[1]^\vee.$$

The morphisms $f\lab$ induce isomorphisms $g\lab \colon F\lalp|_{U\lab} \to F\lbet|_{U\lab}$ which fit in a commutative diagram
\begin{align*}
    \xymatrix{
    F\lalp|_{U\lab} \ar[d]_-{\phi\lalp|_{U\lab}} \ar[r]^-{g\lab} & F\lbet|_{U\lab} \ar[d]^-{\phi\lbet|_{U\lab}} \\
    \bL_{U\lalp/S}^{\geq -1}|_{U\lab} \ar[r] & \bL_{U\lbet/S}^{\geq -1}|_{U\lab}
    }
\end{align*}
Let $\psi\lab = h^1(g_{\beta \alpha}^\vee) \colon \Ob_{\phi\lalp}|_{U\lab} \to \Ob_{\phi\lbet}|_{U\lab}$. 

Since a scaling automorphism $c\lalp \cdot \id \colon E\lalp \to E\lalp$ with $c\lalp \in \Gamma(\oO_{U\lalp}^*)$ induces the identity automorphism of $\RHom(E\lalp, E\lalp)_0$, the cocycle condition \eqref{cocycle 1} implies that 
 \begin{align*}
    \pi_{\gamma \alpha}^* \psi_{\gamma \alpha} \circ  \pi_{\beta \gamma}^* \psi_{\beta \gamma} \circ \pi\lab^* \psi\lab = \id.
\end{align*} 

It is therefore immediate that these data give an almost perfect obstruction theory on $\mM \to S$.

\begin{thm-defi}
Let $X \to S$ be a smooth, proper family of Calabi-Yau threefolds and $\mM \sub \dD^L_{X/S}$ a proper, open and closed substack of the stack of simple perfect complexes on $X \to S$ with determinant $L$. Then $\mM \to S$ admits an almost perfect obstruction theory and a virtual structure sheaf $[\oO_\mM\virt]\in K_0(\mM)$.

When $S = \Spec \bC$ and $X = W$ is a proper Calabi-Yau threefold, for any class $\beta \in K^0(\mM)$ we define the $\beta$-twisted $K$-theoretic Donaldson-Thomas invariant associated to $\mM$ as the number $\chi(\mM, \beta \cdot \oO_\mM\virt)$. 
\end{thm-defi}

\subsection{$K$-theoretic Donaldson-Thomas invariants of semistable objects by partial desingularizations} Beyond $K$-theoretic DT invariants of derived category objects, we can use almost perfect obstruction theories to produce $K$-theoretic generalized DT invariants of sheaves and complexes, using the results of \cite{KLS} and \cite{Sav}. 

Let $\pi \colon X \to S$ be a smooth, projective family of Calabi-Yau threefolds over a quasi-projective smooth base $S$. Consider the moduli stack $\mM = \mM^{\sigma-ss}(\gamma) \to S$ of fibrewise $\sigma$-semistable perfect complexes in $D^b(\Coh X)$ with Chern character $\gamma \in \Gamma(S,Rp_*\bQ)$, and fixed determinant where $\sigma$ is any stability condition satisfying:
\begin{enumerate}
    \item $\mM \to S$ is the truncation of a $(-1)$-shifted symplectic derived Artin stack $\underline{\mM} \to S$.
    \item $\mM \to S$ admits a proper good moduli space $M \to S$, as in \cite{GoodAlper}.
    \item $\mM \to S$ is of finite type.
\end{enumerate}

Using the results of \cite{PTVV}, this includes the following examples:
\begin{enumerate}
    \item Gieseker stability and slope stability for coherent sheaves with any base $S$, as in \cite{HuyLehn}. These are two classical quotient stacks obtained by Geometric Invariant Theory (GIT).
    \item Polynomial stability with base $S$ being a point, as in \cite{Lo, Lo2}. This is a consequence of the recent results in \cite{AlpHalpHein}.
    \item Bridgeland stability with base $S$ a smooth quasi-projective curve, as in \cite{TodaPiya} and \cite{quinticstab}. This follows from the work of \cite{familystab}, which makes use of \cite{AlpHalpHein} as well.
\end{enumerate}

In \cite{KLS} and \cite{Sav}, it is shown that a canonical procedure, inspired by Kirwan's blowup procedure developed in \cite{Kirwan}, produces the following data:
\begin{enumerate}
    \item A Deligne-Mumford stack $\ti{\mM} \to \mM$, proper over $S$, called the Kirwan partial desingularization of $\mM$.
    \item An \'{e}tale cover $\lbrace U\lalp \to \ti{\mM} \rbrace_{\alpha \in A}$ with Kuranishi models $\Lambda\lalp = (V\lalp, F\lalp, \omega\lalp)$.
    \item For each pair of indices $\alpha, \beta$, an \'{e}tale cover $\lbrace T_\gamma \to U\lab \rbrace_{\gamma \in \Gamma}$ such that for $\lambda = \alpha, \beta$ there exist unramified morphisms $\Phi_\lambda \colon V_\gamma \to V\lal$ making the diagrams
    \begin{align*}
    \xymatrix{
        T_\gamma \ar[d] \ar[r] & U\lal \ar[d] \\
        V_\gamma \ar[r] & V\lal
        }
    \end{align*}
    commute and a Kuranishi model $M_\gamma = (V_\gamma, F_\gamma, \omega_\gamma)$ on $T_\gamma$ which is $\Omega$-compatible to $\Lambda\lal$ via $\Phi\lal$. 
\end{enumerate}

We briefly outline the construction for the convenience of the reader in the absolute case when the base $S$ is a point and $\mM$ is a moduli stack of Gieseker semistable sheaves. In this case, $\mM$ is obtained by GIT and is a global quotient stack $\mM = [ X / G]$. 

Since $\mM$ is the truncation of a $(-1)$-shifted symplectic derived stack, the results of \cite{JoyceArt} imply that $\mM$ is a d-critical Artin stack. In particular, for every closed point $x \in \mM$ with (reductive) stabilizer $H$, there exists a smooth affine $H$-scheme $V$, an invariant function $f \colon V \to \bA^1$ and an \'{e}tale morphism
\begin{align} \label{eq1}
[ U / H ] \to \mM,
\end{align}
where $U = (df = 0) \sub V$. Moreover for every two such local presentations, there exist appropriate comparison data which are similar to the case of a d-critical Deligne-Mumford stack in Subsection~\ref{d-crit application}.

We have the following $H$-equivariant 4-term complex
\begin{equation}\label{eq2}
\fh=\mathrm{Lie}(H)\lra T_{V}|_{U} \xrightarrow{d(df)^\vee} F_{V}|_{U}=\Omega_{V}|_{U} \lra \fh^\vee.
\end{equation}
For $u \in U$ with finite stabilizer, this is quasi-isomorphic to a 2-term complex which gives a symmetric perfect obstruction theory of $[U / H]$ and thus of $\mM$ near $u$.

One may then apply Kirwan's partial desingularization procedure, using the notion of intrinsic blowup introduced in \cite{KiemLi}, as adapted in \cite{KLS}, to obtain the Kirwan partial desingularizations $\tilde{X} \to X$ and $\tilde{\mM} := [ \tilde{X} / G ] \to \mM$, which is a proper DM stack. The main idea is to perform a modified blowup of the loci of sheaves with the same reductive stabilizer, starting with the stabilizers of largest dimension and proceeding in decreasing order. 

We may lift the \'{e}tale cover \eqref{eq1} to an \'{e}tale cover
\begin{align} \label{eq3}
[ T / H ] \to \tilde{\mM},
\end{align}
where $T = ( \omega_S = 0 ) \sub S$ for $S$ a smooth affine $H$-scheme and $\omega_S \in H^0(S, F_S)$ an invariant section of an $H$-equivariant vector bundle $F_S$ on $S$. Moreover, there exists an effective invariant divisor $D_S$ such that \eqref{eq2} lifts to a 4-term complex
\begin{equation}\label{eq4}
\fh = \mathrm{Lie}(H) \lra T_{{S}}|_{{T}} \lra F_{{S}}|_{{T}} \lra \fh^\vee(-D_{{S}})
\end{equation}
whose first arrow is injective and last arrow is surjective. Therefore, \eqref{eq4} is quasi-isomorphic to a 2-term complex
\begin{equation}\label{eq5} 
d(F_{S}^\rred):
( d \omega_{S}^\vee )^\vee : T_{[S/H]}|_{{T}} \lra F^\rred_{{S}}|_{{T}},
\end{equation}
where $F^\rred_{{S}}$ is the kernel of the last arrow in \eqref{eq4}. 
Dualizing and taking the quotient by $H$, we get
\begin{equation}\label{eq6} 
d (\omega_{S}^\rred)^\vee : F_{[{S}/H]}^\rred|_{[{T}/H]}^\vee \lra \Omega_{[{S}/H]}|_{[{T}/H]}.
\end{equation}

Taking \'{e}tale slices of $[T/H]$ and using the $H$-equivariant data $(S, F_S^\rred, \omega_S^\rred)$ gives rise to the above \'{e}tale cover $\lbrace U\lalp \to \ti{\mM} \rbrace$ and the Kuranishi models $\Lambda\lalp$.\\

As in the previous subsection, by (3) and Lemma~\ref{lemma 7.2} we see that $E_{\Lambda\lal}|_{T_\gamma}$ and $E_{M_\gamma}$ are isomorphic obstruction theories on $T_\gamma$. Thus $E\lalp|_{T_\gamma}$ and $E\lbet|_{T_\gamma}$ are isomorphic obstruction theories and it is shown in \cite{KLS} that the induced isomorphisms for their obstruction sheaves glue to define a global obstruction sheaf. We thus obtain an almost perfect obstruction theory $\phi$ for $\ti{\mM} \to S$ on the cover $\lbrace U\lalp \to \ti{\mM} \rbrace$.

We can therefore give the following definition.

\begin{thm-defi}
Let $W$ be a smooth, projective Calabi-Yau threefold and $\mM = \mM^{\sigma-ss}(\gamma)$ be the moduli stack of $\sigma$-semistable perfect complexes in $D^b(\Coh W)$ with Chern character $\gamma$, where $\sigma$ is as above.

The Kirwan partial desingularization $\ti{M}$ admits an almost perfect obstruction theory $\phi$ and thus a virtual structure sheaf $[\oO_{\ti{\mM}}\virt] \in K_0(\ti{\mM})$. For any $\beta \in K^0(\mM)$, the $\beta$-twisted $K$-theoretic Donaldson-Thomas invariant via Kirwan blowups of $\mM$ is defined as the number
\begin{align*}
    \mathrm{DTK}^{\mathrm{K-th}}(\mM, \beta) = \chi ( \ti{\mM}, p^* \beta \cdot \oO_{\ti{\mM}}\virt )
\end{align*}
where $p \colon \ti{\mM} \to \mM$ is the natural projection morphism.
\end{thm-defi}

In the relative case, where $X \to S$ is a smooth, projective family of Calabi-Yau threefolds with special fiber $W$ and $\mM \to S$ is the relative moduli stack of $\sigma$-semistable complexes, the fact that the Kirwan partial desingularization construction behaves well in families and the deformation invariance of the virtual structure sheaf of an almost perfect obstruction theory, proved above in Subsection~\ref{deformation invariance}, imply that the $K$-theoretic DTK invariant is invariant under deformation of the Calabi-Yau threefold $W$.

\appendix

\section{Semi-perfect Obstruction Theory} \label{appendix}

In this appendix we review the definition of a semi-perfect obstruction theory and 
prove Proposition \ref{y4}.

Let $U \to S$ be a morphism of finite type, where $U$ is a Deligne-Mumford stack of finite type and $S$ a smooth Artin stack of pure dimension. We first recall the definition of 
perfect obstruction theory \cite{BehFan, LiTian}.

\begin{defi} \emph{(Perfect obstruction theory \cite{BehFan})} \label{Perf obs th}
A (truncated) perfect (relative) obstruction theory consists of a morphism $\phi \colon E \to \bL_{U/S}^{\geq -1}$ in $D^b(\Coh U)$ such that
\begin{enumerate}
\item $E$ is of perfect amplitude, contained in $[-1,0]$.
\item $h^0(\phi)$ is an isomorphism and $h^{-1}(\phi)$ is surjective.
\end{enumerate}
We refer to $\Ob_\phi := \Hone(E^\vee)$ as the obstruction sheaf of $\phi$.
\end{defi}

\begin{defi} \emph{(Infinitesimal lifting problem)} \label{Inf lift prob}
Let $\iota \colon \Delta \to \bar{\Delta}$ be an embedding with $\bar{\Delta}$ local Artinian, such that $I \cdot \fm = 0$ where $I$ is the ideal of $\Delta$ and $\fm$ the closed point of $\bar{\Delta}$. We call $(\Delta, \bar{\Delta}, \iota, \fm)$ a small extension. Given a commutative square
\begin{align}
\xymatrix{
\Delta \ar[r]^g \ar[d]^\iota & U \ar[d]\\
\bar{\Delta} \ar[r] \ar@{-->}[ur]_{\bar{g}} & S
}
\end{align}
such that the image of $g$ contains a point $p \in U$, the problem of finding $\bar{g} \colon \bar{\Delta} \to U$ making the diagram commutative is the ``infinitesimal lifting problem of $U/S$ at $p$".
\end{defi}

\begin{defi} \emph{(Obstruction space)} \label{Obs spaces}
For a point $p \in U$, the intrinsic obstruction space to deforming $p$ is $T_{p, U/ S}^1 := \Hone \left( (\bL_{U/S}^{\geq -1})^\vee \vert_p \right)$. The obstruction space with respect to a perfect obstruction theory $\phi$ is $\OB(\phi,p) := \Hone( E^\vee \vert_p )$.
\end{defi}

Given an infinitesimal lifting problem of $U/S$ at a point $p$, there exists by the standard theory of the cotangent complex a canonical element 
\begin{align}
\omega \left( g, \Delta, \bar{\Delta} \right) \in \Ext^1 \left( g^{\ast} \bL_{U/S}^{\geq -1} \vert_p, I\right) = T^1_{p, U/S} \otimes_\bC I
\end{align} 
whose vanishing is necessary and sufficient for the lift $\bar{g}$ to exist. 

\begin{defi} \emph{(Obstruction assignment)} \label{Obs assignment}
For an infinitesimal lifting problem of $U / S$ at $p$ and a perfect obstruction theory $\phi$ the obstruction assignment at $p$ is the element
\begin{align}
ob_U(\phi,g,\Delta,\bar{\Delta}) = h^1(\phi^\vee) \left( \omega \left( g, \Delta, \bar{\Delta} \right) \right) \in \OB(\phi,p) \otimes_\bC  I.
\end{align}
\end{defi}

\begin{defi} \emph{($\nu$-equivalence)}\label{Same obs assign}
Let $\phi \colon E \to \bL_{U / S}^{\geq -1}$ and $\phi' \colon E' \to \bL_{U / S}^{\geq -1}$ be two perfect obstruction theories and $\psi \colon \Ob_\phi \to \Ob_{\phi'}$ be an isomorphism. We say that the obstruction theories are $\nu$-equivalent if they give the same obstruction assignments via $\psi$, i.e. for any infinitesimal lifting problem of $U/S$ at $p$
\begin{align}
\psi \left( ob_U(\phi,g,\Delta,\bar{\Delta}) \right) = ob_U(\phi',g,\Delta,\bar{\Delta}) \in \OB(\phi',p) \otimes_\bC I.
\end{align}
\end{defi}
We are now ready to give the definition of a semi-perfect obstruction theory.

\begin{defi} \emph{(Semi-perfect obstruction theory \cite{LiChang})} \label{semi-perfect obs th}
Let $X\to S$ be a morphism, where $X$ is a DM stack of finite presentation and $S$ is a smooth quasi-projective scheme. A semi-perfect obstruction theory $\phi$ consists of an \'{e}tale covering $\lbrace U_\alpha \to X \rbrace_{\alpha \in A}$ 
of $X$ and perfect obstruction theories $\phi_\alpha \colon E_\alpha \to \bL_{U_\alpha / C}^{\geq -1}$ such that
\begin{enumerate}
\item For each pair of indices $\alpha, \beta$, there exists an isomorphism \begin{align*}
\psi_{\alpha \beta} \colon \Ob_{\phi_\alpha} \vert_{U_{\alpha\beta}} \lra \Ob_{\phi_\beta} \vert_{U_{\alpha\beta}}
\end{align*}
so that the collection $\lbrace \Ob_{\phi\lalp}, \psi\lab \rbrace$ gives descent data of a sheaf on $X$.
\item For each pair of indices $\alpha, \beta$, the obstruction theories $E_\alpha \vert_{U_{\alpha \beta}}$ and $E_\beta \vert_{U_{\alpha \beta}}$ give the same obstruction assignments via $\psi_{\alpha \beta}$ (as in Definition \ref{Same obs assign}).
\end{enumerate}
\end{defi}

\begin{rmk}
The obstruction sheaves $\lbrace \Ob_{\phi\lalp} \rbrace_{\alpha \in A}$ glue to define a sheaf $\Ob_{\phi}$ on $X$. This is the obstruction sheaf of the semi-perfect obstruction theory $\phi$.
\end{rmk}

We end this paper with a proof of the comparison result of generalized perfect obstruction theories. 

\begin{proof}[Proof of Proposition \ref{y4}]
By Definitions~\ref{semi-perfect obs th} and \ref{almost perfect obs th}, we need to show any two obstruction theories $E\lalp|_{U\lab}$ and $E\lbet|_{U\lab}$, which are part of the data of an almost perfect obstruction theory, give the same obstruction assignments via $\psi\lab$ (cf. Definition~\ref{Same obs assign}). 

Consider an infinitesimal lifting problem (cf. Definition~\ref{Inf lift prob}) of $U\lab/S$ at a point $p$. By definition, there exists $V:= V_\gamma \to U\lab$ \'{e}tale so that $g \colon \Delta \to U$ factors through $V \to U\lab$ and $E\lalp|_V$ and $E\lbet|_V$ are isomorphic and compatible with $\psi$. We then have a commutative diagram
\begin{align*} 
\xymatrix{
    g^* E\lalp^\vee|_p \ar[r]^-{\phi\lalp^\vee|_p} & g^* ( \bL_{U\lalp/S}^{\geq -1} )^\vee |_p  \ar[r] & g^*( \bL_{U\lab/S}^{\geq -1} )^\vee |_p \ar[r] & I[1] \\
    g^* E\lbet^\vee|_p \ar[r]_-{\phi\lbet^\vee|_p} \ar[u]_-{\psi\labc^\vee|_p} & g^* ( \bL_{U\lbet/S}^{\geq -1} )^\vee |_p  \ar[ur] \ar[u] 
    }
\end{align*}
which implies immediately that 
\begin{align*}
    ob_{U\lab}(\phi\lalp,g,\Delta,\bar{\Delta})  & = h^1(\psi\labc^\vee|_p) \left( ob_{U\lab}(\phi\lbet,g,\Delta,\bar{\Delta}) \right) = \\
    & = \psi\lab^{-1}|_p \left( ob_{U\lab}(\phi\lbet,g,\Delta,\bar{\Delta}) \right)
\end{align*}
and hence  $E\lalp|_{U\lab}$ and $E\lbet|_{U\lab}$ give the same obstruction assignments via $\psi\lab$, as desired.
\end{proof}

\bibliography{Master}
\bibliographystyle{alpha}

\end{document}